\documentclass[12pt]{amsart}
\usepackage{amsmath,amsthm,amssymb,bbm,dsfont}
\usepackage{graphicx}
\usepackage{subfig}
\usepackage{cite,float}

\numberwithin{equation}{section}
\newtheorem{theorem}{Theorem}[section]
\newtheorem{lemma}[theorem]{Lemma}

\newtheorem{corollary}[theorem]{Corollary}
\newtheorem{proposition}[theorem]{Proposition}
\theoremstyle{remark}
\newtheorem{remark}[theorem]{Remark}

\makeatletter
\@namedef{subjclassname@2010}{%
	\textup{2010} Mathematics Subject Classification}
\makeatother
\begin{document}
	
	\title{Strongly linearly convex exhaustion of  a class of $\mathbb{C}$-convex domains}
	
	\author[N. Gupta]{Naveen Gupta}
	\address{Department of Mathematics\\
		Netaji Subhas University of Technology\\
		Delhi-110078\\
		India}
	\email{naveen.gupta@nsut.ac.in, ssguptanaveen@gmail.com}

	\begin{abstract}
Let  $D\subseteq \mathbb{C}^n$ be a bounded $\mathbb{C}$-convex domain with 
$C^1$ boundary whose 
outward unit normal admits a modulus of continuity $\omega$ satisfying 
$\lim_{t\to 0^+}\dfrac{\omega(t)}{\sqrt t}$. We prove that $D$ admits an
 increasing exhaustion	by bounded $C^{\infty}$ strongly linearly convex domains. 
 This, in particular, answers a question posed by Azinberg \cite{azin} in affirmative 
 for a class of domains $D$.
	\end{abstract}
	\maketitle
	
	\section{Introduction}
	
	Edigarian in his very recent paper \cite{armen} proved that every bounded 
	$\mathbb{C}$-convex domain $D\subseteq \mathbb{C}^n,\, n\geq 2$ with boundary of class
	 $C^{1,\alpha}, \, 1/2<\alpha \leq 1$ admits an increasing exhaustion by bounded 
	 $C^{\infty}$ strongly linearly convex domains. The genesis of this problem goes back to 
	 Azinberg \cite{azin}, where he posed the following question: \emph{Does every bounded 
	 $\mathbb{C}$-convex domain admit an increasing exhaustion by smooth strictly 
	 $\mathbb{C}$-convex domains?}
	 Recall that a domain is $\mathbb{C}$-convex if its intersection with every affine 
	 complex line is either non-empty of connected and simply connected.
	 
	 D. Jacquet has studied this question for the domains 
	 with $C^2$
	  boundary \cite{jacquet1} and later for domains with $C^1$ boundary but with some extra 
	  interior ball condition \cite{jacquet2}.
	  
	  Another strong form of this quesiton was considered by Pflug and Zwonek, where they 
	  asked if every bounded $\mathbb{C}$-convex domain can be exhausted by strongly
	   linearly convex domains and answered it positively for 
	   the symmetrized bidisc \cite{zwonek}. This viewpoint is relevant for the invariant 
	   distance theory since this gives another route to the equality of Lempert function.
	   
	   The main result of this article is the following theorem, which strengthens 
	   \cite[~Theorem 1.1]{armen}.
	
	\begin{theorem}\label{thm:main}
		Let $D\subseteq \mathbb{C}^n$ be a bounded $\mathbb{C}$-convex domain with 
		$C^1$-boundary. Assume that the outward unit normal $\nu$ to the boundary 
			$\partial D$ admits a modulus 
		of continuity $\omega$ satisfying 
		$\lim_{t\to 0^+}\dfrac{w(t)}{\sqrt t}=0.$ Then, 
		there are bounded domains $D_j\subseteq \mathbb{C}^n$ such that 
		$$\overline{D_j}\subseteq D_{j+1},\,\, \cup_{j=1}^{\infty}D_j=D,$$ 
		where each $D_j$ has $C^{\infty}$ boundary and is strongly linearly convex.

	\end{theorem}

The following immediate corollary is an interesting consequence of this theorem.
\begin{corollary}
	Under the assumption of Theorem \ref{thm:main}, the Carath\'eodory distance equals 
	the Lempert function, and the Carath\'eodory-Reiffen metric equals the 
	Kobayashi-Royden metric on $D$.
\end{corollary}

It is easy to observe that for a bounded $\mathbb{C}$-convex domain whose boundary $C^{1,\alpha},\, 
\alpha >1/2$, its outward unit normal $\nu$ is of class $C^{0,\alpha}$. Thus the function 
$$\omega_{\nu}(t):=\sup\{|\nu(x)-\nu(y)|:|x-y|\leq t\}$$ is a modulus of continuity for 
$\nu$ satifying the hypothesis stated in Theorem \ref{thm:main}. 

Also, a domain satisfying the hypothesis of Theorem \ref{thm:main} is of class 
$C^{1,1/2}$. Therefore, 
our formulation includes domains having boundary of class $C^{1,1/2}$.
	\section{Proof}
	Let us fix some notations first. Let $D$ be as in the hypothesis of Theorem 
	\ref{thm:main}. We set $\delta_{D}(z)=dist(z,\partial D)$ and 
	$h(z)=(\delta_D(z))^2$. Let $R>0$ be large enough so that $\overline{D}\subseteq 
	\mathbb{B}(0,R)$ and we choose $0<c<\dfrac{1}{20(1+R^2)}$. Let $\chi\in C_c^{\infty}
	(\mathbb{C}^n)$ be nonnegative function supported in the unit Euclidean ball and with 
	$\int \chi=1$. For $\tau>0$, we denote $\chi_{\tau}(z)=\dfrac{1}{\tau^{2n}}\chi(z/{\tau}).$
	Then, $\chi_{\tau}$ is supported in the ball $\mathbb{B}(0,\tau)$ and $\int \chi_{\tau}=1.$ 
	For any $\delta>0$, we denote 
	$h_{\delta}=h*\chi_{\tau_{\delta}},$ where $\tau_{\delta}>0$, $\tau_{\delta}$ depends 
	on $\delta$.
	
	For a real valued $C^2$ function $u$ and for $w\in \mathbb{C}^n$, set 
	$$\partial u(z)(w)=\sum_{j=1}^{n}\dfrac{\partial u(z)}{\partial z_j}(w_j)$$ and
	$$L_u(z;w)=\sum_{j,k=1}^{n}\dfrac{\partial^2 u(z)}{\partial z_j \partial 
		\overline{z_k}}(w_j)\overline{w_k},\,
	Q_u(z;w)=\sum_{j,k=1}^{n}\dfrac{\partial^2 u(z)}{\partial z_j \partial 
z_k}(w_j)w_k.$$

For the modulus of continuity $\omega$ of $\nu$(outward unit normal to $\partial D$),	
define 
	$\rho(r):=\sup_{0< t\leq r}\dfrac{\omega(t)}{\sqrt t}$.Then, it is easy to observe 
	that $\rho(r)\to 0$ and that 
	\begin{equation}\label{outward-compare}
	\omega(t)\leq \rho(r)\sqrt t \,\, \mbox{for}\,\, 0<t\leq r. 
	\end{equation}
	
	\begin{proposition}\label{prop:2.1}
		For any $\theta\in \mathbb{R},\, w\in \mathbb{C}^n$ and $t$, at 
		differentiability points $z$ of $h$, we have the following:
		
	$$h(z+te^{i\theta}w)+h(z-te^{i\theta}w)-2h(z)\leq 2t^2\frac{|\partial h(z)(w)|^2}{h(z)}.$$
		\end{proposition}
	
	\begin{remark}
		At differentiability points $z$ of $h$, if $p(z)$ is the nearest boundary point, then 
		$$z=p(z)-\delta_{D}(z)\nu(p(z))\,\, \mbox{and}$$
		\begin{align*}
			\partial h(z)(w)&=<w,z-p(z)>\\
			&=-\delta_{D}(z)<w,\nu(p(z))>.
		\end{align*}
\end{remark}

\begin{lemma}\label{lem:2}
Let $z, \zeta$ be points of differntiability of $h$ satisfying $\dfrac{\delta}{2}\leq 
\delta_{D}(z),\delta_{D}(\zeta)\leq  5\delta,\, |z-\zeta|\leq \dfrac{\delta}{4}$.
 Then for all sufficiently small $\delta>0$, there	there exists constants $A,C>0$ such that 
 $$|\partial h(z)(w)-\partial h(\zeta)(w)|\leq C\left(|z-\zeta|+\delta^2\rho(A\delta)^2
 \right)|w|$$
 for every $w\in \mathbb{C}^n$.
\end{lemma}	

\begin{proof}
	Take $u(z)=z-p(z)=-\delta_{D}(z)\nu(p(z))$, then using the fact that $\nu$ is 
	the unit outward normal it is easy to obtain
\begin{eqnarray*}
		|u(z)-u(\zeta)|\leq|\delta_{D}(z)-\delta_{D}(\zeta)|+\delta_{D}(\zeta)
		\omega(|p(z)-p(\zeta)|) \\
		\leq |z-\zeta|+5\delta \omega(|z-\zeta|+|u(z)-u(\zeta)|).
	\end{eqnarray*}
We have used used the fact that $\omega$ is an increasing function in the last step.       

Observe that $s:=|z-\zeta|+|u(z)-u(\zeta)|\leq |z-\zeta|+10\delta $ , therefore we get, 
$s\leq 11\delta$. Also, using the inequality \ref{outward-compare} for 
$r=11\delta$, we get $\omega(s)\leq 
\rho(11\delta) \sqrt{\delta}.$    This further yields $s\leq 2|z-\zeta|+5\delta \rho(11\delta) 
\sqrt{s}.$ Using the identity $XY\leq \frac 1 2 X^2+\frac 1 2 Y^2$ for $X=\sqrt s$ and 
$Y=5\delta \rho(11\delta)$, the inequality is futher simplified as 
$b\leq 4|z-\zeta|+25\delta^2\rho(11\delta)^2.$ 

Consider 
\begin{align*}
|\partial h(z)(w)-\partial h(\zeta)(w)|&=|<w,u(z)>-<w,u(\zeta)>|\\
&=|<w,u(z)-u(\zeta)>|\\
&\leq |u(z)-u(\zeta)||w|\\
&\leq \left(4|z-\zeta|+25\delta^2\rho(11\delta)^2\right)|w|\\
&\leq C\left(|z-\zeta|+\delta^2\rho(A\delta)^2\right)|w|,                                       
\end{align*}
where $A=11,\, C=25$.
\end{proof}
	
\begin{proposition}\label{prop:2.3}
	Let $h_{\delta}=h*\chi_{\tau_{\delta}}$, where $0<2\tau_{\delta}\leq\dfrac{\delta}{4}.$ 
	Let $U_{\delta}=\{z:\delta<\delta_D(z)<4\delta\}$. Then, on $U_{\delta},$
	$$\frac{|\partial h_{\delta}(z)(w)|^2}{h_{\delta}(z)}-Lh_{\delta}(z;w)-|Qh_{\delta}
(z;w)|\geq -C\delta^2\left(\rho(A\delta)^2+\frac{2\tau_{\delta}}{\delta^2}\right)^2|w|^2.$$
	Furthermore, $|h_{\delta}(z)-h(z)|\leq C\delta\tau_{\delta}$. 

\end{proposition}	
\begin{proof}
Let $z\in U_{\delta},\, \zeta, \eta \in B(z,\tau_{\delta})$ and $\theta\in \mathbb{R}$. 
Since $\delta_D$ is $1$-Lipschitz, it follows that $\delta_D(\zeta)\geq \delta-
\tau_{\delta}$. So, there is some $t_0>0$ depending only on $z,w,\delta$ such that 
$\zeta\pm te^{i\theta}w\in D$ for every $\zeta\in B(z,\tau_{\delta}).,$ for all $|t|<t_0$. 
Using Proposition \ref{prop:2.1}, at almost every $\zeta\in B(z,\tau_{\delta}),$ and 
averaging over $B(z,\tau_{\delta})$, we get 
$$h(z+te^{i\theta}w)+h(z-te^{i\theta}w)-2h(z)\leq 2t^2\left(\frac{|\partial h(w)|^2}{h}*
\chi_{\tau_{\delta}}\right)(z).$$
Multiplying with $\chi_{\tau_{\delta}}(z-\zeta)$ and integrating, we get 
$$h_{\delta}(z+te^{i\theta}w)+h_{\delta}(z-te^{i\theta}w)-2h_{\delta}(z)\leq 
2t^2\left(\frac{|\partial h(w)|^2}{h}*\chi_{\tau_{\delta}}\right)(z).$$
Dividing by $2t^2$ and taking limit as $t\to 0$ (upon using smoothness of $h_{\delta}$), we get
$$Lh_{\delta}(z;w)+Re\left(e^{2i\theta}Qh_{\delta}(z;w)\right)\leq
\left(\frac{|\partial h(w)|^2}{h}*\chi_{\tau_{\delta}}\right)(z).$$
Taking sup over $\theta$, we get

\begin{equation}\label{prop:2.4_eq1}
Lh_{\delta}(z;w)+|Qh_{\delta}(z;w)|\leq
\left(\frac{|\partial h(w)|^2}{h}*\chi_{\tau_{\delta}}\right)(z).
\end{equation}
Let us fix $F(\zeta):=\dfrac{\partial h(\zeta)(w)}{h(\zeta)}$ and consider 
$$F(\zeta)-F(\eta)=\dfrac{\partial h(\zeta)(w)}{h(\zeta)}-\dfrac{\partial h(\eta)(w)}
{h(\eta)}
+\dfrac{\partial h(\eta)(w)}{h(\zeta)}-\dfrac{\partial h(\eta)(w)}{h(\zeta)}.$$
Therefore, 
$$|F(\zeta)-F(\eta)|\leq\dfrac{|\partial h(\zeta)(w)-\partial h(\eta)(w)|}{h(\zeta)}-
|\partial h(\eta)(w)|\dfrac{|h(\zeta)-h(\eta)|}{h(\zeta)h(\eta)}.$$
Using Lemma \ref{lem:2} and the fact that $\partial h(\eta)(w)=-\partial_D(\eta)
<w,\nu(p(\eta))>$, the previous inequlaity gives
\begin{align}\label{prp:2.4_eq2}
|F(\zeta)-F(\eta)|&\leq\dfrac{C\left(2\tau_{\delta}+\delta^2\rho (A\delta)^2\right)|w|}
{\delta^2/4}+\dfrac{5\delta|w|10\delta|\zeta-\eta|}{\left(\delta^2/4\right)^2}\\
&\leq C_1\left(\rho(A\delta)^2+\dfrac{2\tau_{\delta}}{\delta^2}\right)|w|,
\end{align}
for some constant $C_1$. 
Consider the identity 
\begin{align*}
\dfrac{|\partial h(\zeta)(w)|^2}{h(\zeta)}\chi_{\tau_{\delta}}(z-\zeta)-
\dfrac{|\partial h_{\delta}(z)(w)|^2}{h_{\delta}(z)}&=\int h(\zeta)
\left|F(\zeta)-\dfrac{|\partial h_{\delta}(z)(w)|^2}{h_{\delta}(z)}\right|^2
\chi_{\tau_{\delta}}(z-\zeta)d\zeta\\
&=\int h(\zeta)\left|\int (F(\zeta)-F(\eta))d\mu_z(\eta) \right|
\chi_{\tau_{\delta}}(z-\zeta)d\zeta,
\end{align*}
where $d\mu_z(\eta)=\dfrac{h(\eta)\chi_{\tau_{\delta}}(z-\eta)}{h_{\delta}(z)}d\eta$. The 
proposition now follows by using Equation (\ref{prp:2.4_eq2}) alongwith the previous 
equation and then using Equation (\ref{prop:2.4_eq1}).
\end{proof}

\begin{proposition}\label{prop:2.5}
There is a constant $c_*>0$ such that for all sufficiently small $\delta>0$, every 
$z\in U_{\delta}$ and every $w\in \mathbb{C}^n$ satisty the inequality
$$\frac{|\partial \tilde{h}_{\delta}(z)(w)|^2}{\tilde{h}_{\delta}(z)}-
L\tilde{h}_{\delta}(z;w)-|Q\tilde{h}_{\delta}
(z;w)|\geq c_*\delta^2,$$	
where $\tilde{h}_{\delta}=h_{\delta}-C\delta^2(1+|z|^2).$
\end{proposition}
\begin{proof} It is easy to verify that $\tilde{h}_{\delta}>0$. Applying 
	\cite[~Lemma 2.4]{armen} for $\epsilon =C\delta^2$ and $g=h_{\delta}$ and then 
	using Proposition \ref{prop:2.3} for $\tau_{\delta}=\delta^3$, we get
	\begin{align*}\label{prop:2.5_eq1}
	\frac{|\partial \tilde{h}_{\delta}(z)(w)|^2}{\tilde{h}_{\delta}(z)}-
	L\tilde{h}_{\delta}(z;w)-|Q\tilde{h}_{\delta}
	(z;w)|&\geq -C\delta^2\left(\rho(A\delta)^2+2\delta\right)^2
	|w|^2+\dfrac{C\delta^2}{1+R^2}|w|^2.\\
	&=\delta^2|w|^2\left(\dfrac{C}{1+R^2}-C\left(\rho(A\delta)^2+2\delta\right)^2\right).
	\end{align*}
	Since $C\left(\rho(A\delta)^2+2\delta\right)^2\to 0$, therefore for sufficiently small
	 $\delta>0$, we have 
	 $$C\left(\rho(A\delta)^2+2\delta\right)^2\leq \dfrac{C}{1+R^2}.$$ 
	 Thus the result holds for $c^*=\dfrac{C}{2(1+R^2)}.$
\end{proof}

Now that all the supporting Lemmas and Propositions have been taken care of, the 
proof of the main Theorem \ref{thm:main} proceeds as in the proof of 
\cite[~Theorem 1.1]{armen} by taking $\tau_{\delta}=\delta^3$. In particular, the 
same construction for $D_j$ as $G_{\delta_j}$ yields the increasing exhaustion 
by bounded $C^{\infty}$ stronly linearly convex domains of $D.$


\end{document}